\documentclass[12pt,leqno]{article}

\usepackage[lite]{amsrefs}
\usepackage{latexsym,enumerate}

\usepackage[notref, notcite]{}
\usepackage{amssymb, amsfonts, eucal,amsmath,amsthm}

\usepackage{color}
\definecolor{added}{rgb}{0, 0, 1}
\definecolor{deleted}{rgb}{1, 0, 0}
\newtheorem{theorem}{Theorem}[section]

\newtheorem{lemma}[theorem]{Lemma}
\newtheorem{proposition}[theorem]{Proposition}
\newtheorem{remark}[theorem]{Remark}
\newtheorem{definition}[theorem]{Definition}

\newtheorem{corollary}[theorem]{Corollary}

\newcommand{\sect}[1]{\section{#1} \setcounter{equation}{0} }

\newcounter{ca}
\newcommand{\norm}[2]{\left\|#1\right\|_{#2}}

\newcommand{\Pn}{\mathbb P_n}

 \newcommand{\ec}{\end{comment}}
\newcommand{\bc}{ \begin{comment}
 }

\newcommand{\F}{{\mathcal F}}
\newcommand{\E}{{\mathcal E}}
\newcommand{\G}{{\mathcal G}}
\newcommand{\I}{{\mathcal I}}

\newcommand{\andd}{\quad\mbox{\rm and}\quad}

\newcommand\e{{\varepsilon}}

\newcommand\w{{\omega}}

\def\be  {\begin{equation}}
\def\ee  {\end{equation}}

\newenvironment{comment}[2]
{\bgroup\vspace{7pt}
\begin{tabular}{|p{5in}|}
\hline \qquad \bf \footnotesize Comment -- to be deleted in the final version \\
\hline
\quad\sl\footnotesize #1#2} {\\ \hline \end{tabular}
\vspace{7pt}\indent\egroup}

\def\updots{\mathinner{\mkern
1mu\raise 1pt \hbox{.}\mkern 2mu \mkern 2mu \raise
4pt\hbox{.}\mkern 1mu \raise 7pt\vbox {\kern 7 pt\hbox{.}}} }

\def \esssup{\mathop{\rm ess\: sup}\nolimits}

\newcommand{\B}{\mathbb B}
\newcommand{\C}{C}

\newcommand{\R}{\mathbb R}

\newcommand{\N}{\mathbb N}

\renewcommand{\a}{\alpha}
\renewcommand{\b}{\beta}
\newcommand{\ineq}[1]{{\rm(\ref{#1})}}

\newcommand{\ie}{{\em i.e., }}
\newcommand{\eg}{{\em e.g. }}

\newcommand{\bpic}{
\begin{center}
}

\newcommand{\epic}{
\endpspicture
\end{center}
}

\newcommand{\st}{\;\; \big| \;\;}

\renewcommand{\L}{L}
\newcommand{\Lp}{\L_p}
\newcommand{\Lq}{\L_q}
\newcommand{\Poly}{\mathbb P}
 \newcommand{\AC}{\mathrm{AC}}
  \newcommand{\loc}{\mathrm{loc}}
\newcommand{\wkr}{\w_{k,r}^\varphi}
\newcommand{\wkrav}{\w_{k,r}^{*\varphi}}
 \newcommand{\Dom}{{\mathfrak{D}}}

\newcommand{\ddelta}{\mu(\delta)}

 \newcommand{\wt}{{\mathcal{W}}}
\newcommand{\W}{{\mathcal{W}}}
\newcommand{\Z}{{\mathcal Z}}

\newcommand{\thm}[1]{Theorem~\ref{#1}}
\newcommand{\lem}[1]{Lemma~\ref{#1}}
\newcommand{\cor}[1]{Corollary~\ref{#1}}

 \newcommand{\Lpab}{\L_p^{\alpha,\beta}}
 \newcommand{\wab}{w_{\a,\b}}
 \newcommand{\weight}{\wt_{kh}^{r/2+\a,r/2+\b}}

\title{{\sc On moduli of smoothness with Jacobi weights}
\thanks{{\sc UDC:} 517.5 } }

\author{K. A.  Kopotun\thanks{Department of Mathematics, University of
Manitoba, Winnipeg, Manitoba, R3T 2N2, Canada ({\tt
Crimea\_is\_Ukraine@shaw.ca}). Supported by NSERC of Canada.} ,
D. Leviatan\thanks{Raymond and Beverly Sackler School of Mathematical
Sciences, Tel Aviv University, Tel Aviv 6139001, Israel ({\tt
leviatan@post.tau.ac.il}).}\ \ and I. A. Shevchuk\thanks
{Faculty of Mechanics and Mathematics, Taras Shevchenko
National University of Kyiv, 01033 Kyiv, Ukraine ({\tt
shevchuk@univ.kiev.ua}).}}

\begin{document}

\maketitle

\centerline{\sl\small Dedicated to Academician Anatoly Samoilenko  on the occasion of his 80th birthday}

\abstract{
The main purpose of this paper is to introduce moduli of smoothness with Jacobi weights $(1-x)^\a(1+x)^\b$ for functions in the Jacobi weighted $\Lp[-1,1]$, $0<p\leq \infty$, spaces.
These moduli are used to characterize the smoothness of (the derivatives of)  functions in the weighted $\Lp$ spaces.
If $1\le p\le\infty$, then these moduli are equivalent to certain weighted $K$-functionals  (and so they are equivalent
to certain weighted  Ditzian-Totik  moduli of smoothness for these $p$), while for $0<p<1$ they are equivalent to certain  ``Realization functionals''.
 }

\sect{Introduction and main results}

The main purpose of this paper is to introduce moduli of smoothness with Jacobi weights $(1-x)^\a(1+x)^\b$ for functions in the Jacobi weighted $\Lp[-1,1]$, $0<p\leq \infty$, spaces.
These moduli generalize the moduli that were recently introduced by the authors in \cites{kls, kls1}  in order to characterize the smoothness of (the derivatives of)  functions in the ordinary (unweighted) $\Lp$ spaces.

For a measurable function $f:[-1,1]\mapsto\R$ and an interval $I\subseteq [-1,1]$, we use the usual notation
$\norm{f}{\Lp(I)}:= (\int_I |f(x)|^p\,dx )^{1/p}$, $0<p<\infty$, and  $\|f\|_{L_\infty(I)}:=\esssup_{x\in I}|f(x)|$.
%\[
%L_{p}(I):=\left\{f\mid \norm{f}{\Lp(I)}<\infty \right\},\quad  p\le\infty.
%\]
For a weight function $w$, we let
$L_{w,p}(I):=\{f\mid \norm{wf}{\Lp(I)}<\infty \}$,
and, for $f\in L_{w,p}(I)$,
we denote by $E_n(f,I)_{w,p}:=\inf_{p_n\in\Pn}\|w(f-p_n)\|_{\Lp(I)}$,
the error of best weighted approximation of $f$ by polynomials in $\Pn$, the set of algebraic polynomials of degree strictly less than $n$.
For $I=[-1,1]$, we denote $\norm{\cdot}{p}:= \norm{\cdot}{\Lp[-1,1]}$, $L_{w,p} := L_{w,p}[-1,1]$, $E_n(f)_{w,p} := E_n(f,[-1,1])_{w,p}$, etc. Finally, denote
\[
\varphi(x):=\sqrt{1-x^2}.
\]
\begin{definition} For $r\in\N_0$ and  $0<p \le \infty$, denote $\B^0_p(w) :=L_{w,p}$ and
\[
\B_p^r(w):=\left\{\,f\,|\,f^{(r-1)}\in AC_{loc}(-1,1)\quad\text{and}\quad\varphi^rf^{(r)}\in L_{w,p} \right\}, \quad r\ge 1,
\]
where $AC_{loc}(-1,1)$ denotes the set of  functions which are locally absolutely continuous in $(-1,1)$.
\end{definition}

Now, define
\[
J_p := \begin{cases}
(-1/p, \infty), & \mbox{\rm if } p<\infty , \\
[0,\infty), & \mbox{\rm if } p=\infty ,
\end{cases}
\]
let
\[
w_{\a,\b}(x):=(1-x)^\a(1+x)^\b  , \quad \a,\b \in J_p ,
\]
be the Jacobi weights, and denote $\Lpab := \L_{\wab,p}$.

Also denote
\[
\wt_\delta^{\xi,\zeta} (x):=  (1-x-\delta\varphi(x)/2)^\xi
(1+x-\delta\varphi(x)/2)^\zeta.
\]
Note that $\wt_0^{\a,\b} (x) = \wab(x)$, $\wt_0^{1/2,1/2} (x) = \varphi(x)$ and, if $\xi,\zeta \geq 0$, $\wt_\delta^{\xi,\zeta} (x) \leq w_{\xi,\zeta}(x)$.

For $k\in\N$ and $h\geq 0$, let
\[
\Delta_h^k(f,x; J):=\left\{
\begin{array}{ll} %\ds
\sum_{i=0}^k  \binom{k}{i}
(-1)^{k-i} f(x-\frac{kh}2+ih),&\mbox{\rm if }\, [x-\frac{kh}2, x+\frac{kh}2]  \subseteq J \,,\\
0,&\mbox{\rm otherwise},
\end{array}\right.
\]
be the $k$th symmetric difference, and $\Delta_h^k(f,x) := \Delta_h^k(f,x; [-1,1])$.

We introduce the following definition, which for $\a,\b=0$, was given in \cite{kls}*{Definition 2.2} (for $\a,\b=0$ and $p=\infty$ see the earlier \cite{sh}*{Chapter 3.10}).

\begin{definition} \label{maindefinition}
For $k,r\in\N$ and
$f\in \B^r_p(w_{\a,\b})$, $0< p\le\infty$,  define
\be \label{wkrdefinition}
\wkr(f^{(r)},t)_{\a,\b,p}:=\sup_{0\leq h\leq t}
\norm{ \weight(\cdot) \Delta_{h\varphi(\cdot)}^k (f^{(r)},\cdot)}p.
\ee
\end{definition}

For $\delta>0$, denote (see \cite{kls})
\begin{align*}
\Dom_\delta:=&\left\{x\st1-\delta\varphi(x)/2\geq|x|
\right\}\setminus\{\pm1\}\\=&\left\{x\st|x|\leq
\frac{4-\delta^2}{4+\delta^2}\right\}=[-1+\ddelta,1-\ddelta],
\end{align*}
where
\[
\ddelta:=2\delta^2/(4+\delta^2).
\]
Observe that $\Dom_{\delta_1}\subset\Dom_{\delta_2}$ if $\delta_2<\delta_1\le2$, and that $\Dom_\delta=\emptyset$ if $\delta>2$. Also note that
${\Delta}_{h\varphi(x)}^k(f,x)$ is defined to be identically 0 if
$x\not\in\Dom_{kh}$ and that  $\wt_\delta^{r/2+\a,r/2+\b}$ is well defined on
$\Dom_\delta$ (except perhaps at the endpoints where it may be infinite).

Hence,
\be\label{dom}
\wkr(f^{(r)},t)_{\a,\b,p}=\sup_{0<h\leq t}\norm{\weight(\cdot)
\Delta_{h\varphi(\cdot)}^k(f^{(r)},\cdot)}{L_p(\Dom_{kh})}
\ee
and
\be\label{larget}
\wkr(f^{(r)},t)_{\a,\b,p}=\wkr(f^{(r)},2/k)_{\a,\b,p},\quad\mbox{\rm for }t\geq 2/k .
\ee

In a forthcoming paper \cite{stok}, we will prove   Whitney-,  Jackson- and Bernstein-type  theorems for the Jacobi weighted approximation of functions in the above spaces by algebraic polynomials. Thus, we get a constructive  characterization of the smoothness classes with respect to these moduli by means of the degrees of approximation. This implies, in particular, that these moduli are the right measure of smoothness to be used while investigating constrained weighted approximation (see \eg  \cites{kls-umzh, glsw, klps}).

We will show that, for %$1\le p<\infty$ and
$r/2+\a, r/2+\b\geq 0$, our moduli are equivalent to the following weighted averaged moduli.

\begin{definition}
For $k\in\N$, $r\in\N_0$ and  $f\in \B^r_p(w_{\a,\b})$, $0< p<\infty$,    the $k$th weighted averaged modulus of smoothness of $f$ is defined as
\[
\wkrav(f^{(r)},t)_{\a,\b,p}
:=\left(\frac1t\int_0^t\int_{\Dom_{k\tau}}
| \wt^{r/2+\a,r/2+\b}_{k\tau}(x)\Delta^k_{\tau\varphi(x)}(f^{(r)},x)|^p\,dx\,d\tau
\right)^{1/p} .
\]
If $p=\infty$ and   $f\in\B^r_\infty(w_{\a,\b})$, we write
\[
\wkrav(f^{(r)},t)_{\a,\b,\infty}:=\wkr(f^{(r)},t)_{\a,\b,\infty}\,.
\]
\end{definition}

Clearly, %for $0< p <\infty$,
\be\label{ineq}
\omega_{k,r}^{*\varphi}(f^{(r)},t)_{\a,\b,p}\le\wkr(f^{(r)},t)_{\a,\b,p} ,\quad t>0,\quad 0<p\le\infty.
\ee

We now define the weighted $K$-functional as well as the ``Realization functional'' as follows.

\begin{definition} \label{kfunctionals}
For $k\in\N$, $r\in\N_0$ and $f\in\B^r_p(w_{\a,\b})$, $0< p\le\infty$,  define
\begin{align*}
\lefteqn{ K^\varphi_{k,r}(f^{(r)},t^k)_{\a,\b,p} } \\
&\quad :=\inf_{g\in\B^{k+r}_p(\wab)}   \left\{ \norm{\wab \varphi^r (f^{(r)}-g^{(r)})}{p}
+t^k \norm{\wab\varphi^{k+r}g^{(k+r)}}{p} \right\}
\end{align*}
and
\begin{align*}
\lefteqn{  R^\varphi_{k,r}(f^{(r)},n^{-k})_{\a,\b,p} }\\
&\quad := \inf_{P_n\in\Pn} \left\{  \norm{\wab \varphi^r (f^{(r)}-P_n^{(r)})}{p} + n^{-k} \norm{\wab \varphi^{k+r} P_n^{(k+r)}}{p} \right\} .
\end{align*}
\end{definition}
 Clearly, $K^\varphi_{k,r}(f^{(r)},n^{-k})_{\a,\b,p} \leq R^\varphi_{k,r}(f^{(r)},n^{-k})_{\a,\b,p}$, $n\in\N$.
Note that, as is rather well known,   $K$-functionals are not the right measure of smoothness if $0<p<1$, since they may become identically zero.

Throughout this paper, all constants $c$   may depend only on  $k$, $r$, $p$, $\a$ and $\b$, but are independent of the function as well as the important parameters $t$ and $n$. The constants $c$ may be different even if they appear in the same line.

%\sect{Main results}

Our first main result in this paper is the following theorem.
It is a corollary of Lemma~\ref{upper} and the sequence of estimates \ineq{mainlower}.

\begin{theorem}\label{thm1.4}
If $k\in\N$, $r\in\N_0$, $r/2+\a\geq 0$, $r/2+\b\geq 0$, $1\leq p\leq \infty$ and $f\in\B_p^r(\wab)$,  then
 there exists $N\in\N$ depending on $k$, $r$, $p$, $\a$ and $\b$, such that
 for all $0<t\leq 2/k$ and $n\in\N$ satisfying
 %$n\geq N$ and $n\sim 1/t$,
 $\max\{N, c_1/t\}\le n\le c_2/t$,
\begin{align}   \label{maintheorem}
  K^\varphi_{k,r}(f^{(r)},t^k)_{\a,\b,p} &\leq c R^\varphi_{k,r}(f^{(r)},n^{-k})_{\a,\b,p} \leq  c \wkrav(f^{(r)},t)_{\a,\b,p}\\
 & \leq c \wkr(f^{(r)},t)_{\a,\b,p} \leq c  K^\varphi_{k,r}(f^{(r)},t^k)_{\a,\b,p} \nonumber  ,
\end{align}
 where constants $c$ may depend only on $k$, $r$, $p$, $\a$, $\b$ as well as $c_1$ and $c_2$.
\end{theorem}

\begin{remark} Clearly, $K^\varphi_{k,r}(f^{(r)},t^k)_{\a,\b,p} \leq \norm{\wab \varphi^r f^{(r)}}{p}<\infty$, for all $f\in\B_p^r(\wab)$, and
it follows from {\rm \thm{smallab}} that, if  $r/2+\a <0$ or/and $r/2+\b<0$, then there exists a function $f\in\B_p^r(\wab)$ such that
$\wkr(f^{(r)},t)_{\a,\b,p} = \infty$, for all $t>0$. Hence, {\rm \thm{thm1.4}} is not valid if  $r/2+\a <0$ or/and $r/2+\b<0$.
\end{remark}

We can somewhat simplify the statement of \thm{thm1.4} if we remove the realization functional $R^\varphi_{k,r}$ from \ineq{maintheorem}.

\begin{corollary}\label{maincorollary}
If $k\in\N$, $r\in\N_0$, $r/2+\a\geq 0$, $r/2+\b\geq 0$, $1\leq p\leq \infty$ and $f\in\B_p^r(\wab)$,  then,
 for all $0<t\leq 2/k$,
\begin{align*}  % \label{maincor}
  K^\varphi_{k,r}(f^{(r)},t^k)_{\a,\b,p} &\leq     c \wkrav(f^{(r)},t)_{\a,\b,p}
  \leq c \wkr(f^{(r)},t)_{\a,\b,p}
    \leq c  K^\varphi_{k,r}(f^{(r)},t^k)_{\a,\b,p}   .
\end{align*}
%where constants $c$ may depend only on $k$, $r$, $p$, $\a$ and $\b$.
\end{corollary}

In the case $0<p<1$, we have the following result on the equivalence of the moduli and Realization functionals. It is a corollary of \thm{thm46} that will be proved in Section~\ref{seclower}.

\begin{theorem} \label{pless1}
Let  $k\in\N$, $r\in\N_0$,  $0<p<1$, $r/2+\a\geq 0$, $r/2+\b\geq 0$ and $f\in\B_p^r(\wab)$. Then there exist $N\in\N$ and $\vartheta>0$ depending on $k$, $p$, $\a$ and $\b$, such that, for any $\vartheta_1\in (0,\vartheta]$, $n\geq N$, $\vartheta_1/n\leq t \leq \vartheta/n$, we have
\[
R^\varphi_{k,r}(f,n^{-k})_{\a,\b,p} \sim \w_{k,r}^{*\varphi}(f^{(r)},  t)_{\a,\b,p}\sim \w_{k,r}^{\varphi}(f^{(r)},  t)_{\a,\b,p}.
\]
\end{theorem}
Here, as usual, by $a(t)\sim b(t)$, $t\in T$,  we mean  that there exists a positive constant $c_0$  such that $c_0^{-1}a(t)\le b(t)\le c_0 a(t)$, for all $t \in T$.

Note that it follows from \thm{pless1} that, for sufficiently small $t_1, t_2>0$ such that $t_1\sim t_2$,
\[
\w_{k,r}^{*\varphi}(f^{(r)},  t_1)_{\a,\b,p}\sim \w_{k,r}^{\varphi}(f^{(r)},  t_1)_{\a,\b,p}
\sim \w_{k,r}^{*\varphi}(f^{(r)},  t_2)_{\a,\b,p}\sim \w_{k,r}^{\varphi}(f^{(r)},  t_2)_{\a,\b,p} .
\]
If $1\leq p \leq \infty$, we can say a bit more.
\thm{thm1.4} and the (obvious) monotonicity of $\wkr(f^{(r)},t)_{\a,\b,p}$, with respect to $t$,
immediately yield the following quite useful property which is not easily seen from Definition~\ref{maindefinition}.

\begin{corollary} \label{corhmod}
Let $k\in\N$, $r\in\N_0$, $r/2+\a\ge0$, $r/2+\b\ge0$,  $1\le p\le\infty$,
  $f\in\B^r_p(w_{\a,\b})$ and $\lambda \ge 1$.  Then, for all $t>0$,
 \be \label{hmod}
 \wkr(f^{(r)},\lambda t)_{\a,\b,p} \leq c \lambda^k \wkr(f^{(r)},t)_{\a,\b,p}  .
 \ee
\end{corollary}

By virtue of \ineq{ko} the following result is an immediate consequence of \cor{maincorollary}.

%However, we emphasize that  \ineq{aver} is an integral part of our proof, and so, in fact, a part of \ineq{ko} is also used in our proof of \thm{thm1.4}.

\begin{theorem} \label{dtequiv}
Let $k\in\N$, $r\in\N_0$, $r/2+\a\ge0$, $r/2+\b\ge0$, and $1\le p\le\infty$.
 If $f\in\B^r_p(w_{\a,\b})$, then, for some $t_0>0$ independent of $f$ and $t$,
\be\label{claim}
\wkr(f^{(r)},t)_{\a,\b,p}\sim\omega^k_\varphi( f^{(r)},t)_{w_{\a,\b}\varphi^r,p}, \quad 0<t\leq t_0,
\ee
where the weighted DT moduli $\omega^k_\varphi( g,\cdot)_{w,p}$ are defined in \ineq{dtmod}.
\end{theorem}

It was shown in \cite{kls1}*{Theorem 5.1} that, for $\xi,\zeta\ge 0$ and $g\in B_p^1(w_{\xi,\zeta})$,
\[
\omega^{k+1}_\varphi( g,t)_{w_{\xi,\zeta},p} \le ct \omega^k_\varphi( g',t)_{w_{\xi,\zeta}\varphi,p} , \quad t>0.
\]
Letting  $\xi := r/2+\a$, $\zeta := r/2+\b$, $g:= f^{(r)}$, using  the fact that
$f^{(r)}\in B_p^1(w_{r/2+\a,r/2+\b})$ if and only if $f \in B_p^{r+1}(\wab)$, by virtue of \ineq{claim},   as well as \ineq{hmod} if  $t$ is ``large''(\ie if $t > t_0$), we immediately get the following result.

\begin{lemma} \label{wkrplus1}
Let $k\in\N$, $r\in\N_0$,   $r/2+\a\geq 0$, $r/2+\b\ge 0$, and $1\le p\le\infty$.  If $f\in\B^{r+1}_p(w_{\a,\b})$, then
\be \label{rplusineq}
\w_{k+1,r}^\varphi (f^{(r)},t)_{\a,\b,p}\le c t \w_{k,r+1}^\varphi(f^{(r+1)},t)_{\a,\b,p}, \quad t>0.
\ee
\end{lemma}

Finally, the following lemma follows from \cite{dt}*{Theorem 6.1.4} using \ineq{claim}.

\begin{lemma}\label{hierarchy2}
Let $k\in\N$, $r\in\N_0$,   $r/2+\a\geq 0$, $r/2+\b\ge 0$, and $1\le p\le\infty$.  If $f\in\B^{r}_p(\wab)$, then
\[
\w_{k+1,r}^\varphi(f^{(r)},t)_{\a,\b,p} \le c \wkr(f^{(r)},t)_{\a,\b,p}  , \quad t>0.
\]
\end{lemma}

\sect{Hierarchy of $B_p^r(\wab)$, (un)boundedness of the moduli and their convergence to $0$}\label{sec2}

Without special references we  use the following evident inequalities:
\[
(1-x)\le2(1-u) \andd (1+x)\le2(1+u) , \quad \text{if} \quad u\in \left[\min\{0,x\}, \max\{0,x\} \right] ,
\]
and
\[
\varphi(x)\le\varphi(u),\quad\text{if}\quad|u|\le|x|\le1 .
\]

Also (see \cite{kls}*{Proposition 3.1(iv)}),
\be \label{derphi}
|\varphi'(x)| \le 1/\delta, \quad \text{for}\quad   x\in\Dom_\delta.
\ee

First we show the hierarchy between the $\B_p^r(w_{\a,\b})$, $r\ge0$, spaces. Namely,

\begin{lemma}\label{hierarchy}
Let $r\in\N_0$, $1\leq p\leq \infty$ and $r/2+\a, r/2+\b \in J_p$. Then,
\be \label{incl}
\B_p^{r+1}(\wab)\subseteq\B_p^r(\wab).
\ee
Moreover, in the case $p=\infty$, if $r/2+\a>0$ and $r/2+\b >0$, then, additionally,
\be \label{inclC}
f\in \B_\infty^{r+1}(\wab) \quad \Longrightarrow \quad \lim_{x\to\pm1} \wab(x)\varphi^r(x)f^{(r)}(x)=0.
\ee
%where ({\bf adding this to make sure that we do not forget about this later})
%\[
%\C^0_\infty(w):=\{\,f\,\mid \, f\in\L_\infty^{\a,\b} \andd
%\lim_{x\to\pm1}\wab (x) f(x)=0\}.
%\]
\end{lemma}

\begin{remark}
Note that we may not relax the condition  $r/2+\a, r/2+\b >0$ in order to guarantee \ineq{inclC}. Indeed, if $\a=-r/2$, for example, then
the function $g(x):=x^r$ is certainly in $\B_\infty^{r+1}(\wab)$ %as well as $\C_\infty^{r+1}(\wab)$, but not in $\C_\infty^{r}(\wab)$ since
but $\lim_{x\to 1} \wab(x) \varphi^r(x) g^{(r)}(x) \neq 0$.

The same example shows that we may not relax the condition $r/2+\a, r/2+\b \in J_p$ in order to guarantee \ineq{incl}, since $\norm{\wab \varphi^r g^{(r)}}{p} = \infty$ if this condition is not satisfied, so that $g\not\in\B_p^r(\wab)$.
\end{remark}

\begin{remark}
For any $r\in\N_0$ and $\a,\b\in\R$, \ineq{incl} is not valid if $0<p<1$. For example, suppose that $f$ is such that
\[
f^{(r)}(x) = \sum_{n=1}^\infty g_n(x) ,
\]
where, for each $n\in\N$,
\[
g_n (x) :=
\begin{cases}
 \frac{H_n}{\e_n}\left( x +1- \frac{1}{n+1} \right), & \text{if}\;  \frac{1}{n+1} <  x+1 \leq  \frac{1}{n+1} +\e_n, \\
H_n , & \text{if}\;  \frac{1}{n+1} +\e_n < x+1 \leq  \frac 1n -\e_n, \\
\frac{H_n}{\e_n}\left( \frac 1n  - x -1 \right), & \text{if}\;  \frac 1n -\e_n <  x+1 \le  \frac 1n ,\\
0, & \text{otherwise,}
\end{cases}
\]
$H_n := n^{r/2+\b+1/p}$,  $\e_n := c_0 n^{-2/(1-p)}$, and  $c_0>0$ is a constant depending only on $p$ that guarantees that $4\e_n n(n+1) <1$, for all $n\in\N$.  Then $f^{(r)}\in \AC_\loc(-1,1)$ and
\[
\norm{\wab \varphi^r f^{(r)}}{p}^p   =   \sum_{n=1}^\infty  \norm{\wab \varphi^r g_n}{p}^p
\geq c \sum_{n=1}^\infty \frac{1}{n^{(r/2+\b)p}} H_n^p n^{-2} = c \sum_{n=1}^\infty n^{-1} = \infty .
\]
Hence, $f\not\in B_p^r(\wab)$. At the same time,
\begin{align*}
\norm{\wab \varphi^{r+1} f^{(r+1)}}{p}^p &=  \sum_{n=1}^\infty  \norm{\wab \varphi^{r+1} g_n'}{p}^p
\leq c \sum_{n=1}^\infty \frac{1}{n^{((r+1)/2+\b)p}} \left( H_n \e_n^{-1} \right)^p \e_n \\
& =   c \sum_{n=1}^\infty n^{1-p/2} \e_n^{1-p} =c \sum_{n=1}^\infty n^{-1-p/2} < \infty ,
\end{align*}
so that $f\in B_p^{r+1}(\wab)$.
\end{remark}

\begin{proof}[Proof of $\lem{hierarchy}$] The proof follows along the lines of \cite{kls}*{Lemma 3.4} with some modifications, we bring it here for the sake of completeness.
Let $g\in\B_p^{r+1}(w_{\a,\b})$, and assume, without loss of generality, that $g^{(r)}(0)=0$
and that $\b\ge\a$. For convenience, denote $A_p:=\norm{\wab \varphi^{r+1}g^{(r+1)}}{p}$.

First, if $p=\infty$, then $A_\infty < \infty$ and
\begin{align}
w_{\a,\b}(x)\varphi^r(x)\bigl|g^{(r)}(x)\bigr|&=w_{\a,\b}(x)\varphi^r(x)\left|\int_0^xg^{(r+1)}(u)\,du\right|\\ \nonumber
%&=w_{\a,\b}(x)\varphi^r(x)\left|\int_0^xw^{-1}_{\a,\b}(u)\varphi^{-r-1}(u)w_{\a,\b}(u)\varphi^{r+1}(u)g^{(r+1)}(u)\,du\right|\\ \nonumber
&\le A_\infty
w_{\a,\b}(x)\varphi^r(x)\left|\int_0^xw^{-1}_{\a,\b}(u)\varphi^{-r-1}(u)\,du\right|\\ \nonumber
& \le 2^{\b-\a}A_\infty
\varphi^{r+2\a}(x)\left|\int_0^x\varphi^{-r-1-2\a}(u)\,du\right|\\ \nonumber
&= 2^{\b-\a}A_\infty
\varphi^{r+2\a}(x)\int_0^{|x|}\varphi^{-r-1-2\a}(u)\,du\\ \nonumber
&\le 2^{\b-\a}A_\infty
\int_0^{|x|}\varphi^{-1}(u)\,du\\ \nonumber
&\le2^{\b-\a}A_\infty
\int_0^1\varphi^{-1}(u)\,du\\ \nonumber
& =\pi2^{\b-\a-1}A_\infty.
\end{align}
Hence, $g\in \B_\infty^r(w_{\a,\b})$, and \ineq{incl} is proved if $p=\infty$.

In order to prove \ineq{inclC} we need to show that, if $r/2+\a, r/2+\b >0$, then
\be\label{limit}
\lim_{x\to\pm1}w_{\a,\b}(x)\varphi^r(x)g^{(r)}(x)=0.
\ee
(Note that we are still not losing generality by assuming that $g^{(r)}(0)=0$.)
We put $\varepsilon:=\min\{r+2\alpha,1\}>0$ and note that
$$
\int_0^x\frac1{\varphi^{2}(u)}\,du=\frac12\ln\frac{1+x}{1-x}.
$$
Therefore,
\begin{align*}
w_{\a,\b}(x)\varphi^r(x)|g^{(r)}(x)|&\le 2^{\b-\a}A_\infty \varphi^\varepsilon (x)\int_0^{|x|}\frac1{\varphi^{1+\varepsilon}(u)}\,du \\
& \le 2^{\b-\a}A_\infty \varphi^\varepsilon (x)\int_0^{|x|}\frac1{\varphi^{2}(u)}\,du\\
&=2^{\b-\a}A_\infty \varphi^\varepsilon(|x|)\ln\frac{1+|x|}{1-|x|}\to 0,\quad |x|\to 1,
\end{align*}
 and \ineq{limit} is proved.

Now let $1\le p<\infty$ and $q:=p/(p-1)$. Then, denoting
\[
\left| \int_0^x |G(u)|^q du \right|^{1/q}:= \sup_{u\in \left[ \min\{0,x\}, \max\{0,x\} \right]  } |G(u)|
\]
if   $q=\infty$, we have by H\"older's inequality
\begin{align*}
  \norm{\wab \varphi^{r}g^{(r)}}{p}^p
&=   \int_{-1}^1w^p_{\a,\b}(x)\varphi^{rp}(x)\left|\int_0^xg^{(r+1)}(u)\,du\right|^p\,dx \\
&\le   \int_{-1}^1w^p_{\a,\b}(x)\varphi^{rp}(x)
\left|\int_{0}^{x}w^{-q}_{\a,\b}(u)\varphi^{-(r+1)q}(u)\,du\right|^{p/q}\\
&   \quad\times\left|\int_{0}^{x}|w_{\a,\b}(u)\varphi^{r+1}(u)g^{(r+1)}(u)|^p\,du\right|\,dx\nonumber\\
&\le    A_p^p
\int_{-1}^1w^p_{\a,\b}(x)\varphi^{rp}(x)\left|\int_{0}^{x}w^{-q}_{\a,\b}(u)\varphi^{-(r+1)q}(u)du\right|^{p/q}dx\\
%&\le A^p
%\int_{-1}^1w^p_{\a,\b}(x)\varphi^{rp}(x)\left(\int_{-|x|}^{|x|}w^{-q}_{\a,\b}(u)\varphi^{-(r+1)q}(u)du\right)^{p/q}dx\\
%&=&  A_p^p
%\int_{-1}^1(1+x)^{(\b-\a)p}\varphi^{rp+2\a p}(x)\left|\int_{0}^{x}(1+u)^{(\a-\b)p}\varphi^{-(r+1)q-2\a q}(u)du\right|^{p/q}dx\\
&\le    2^{(\b-\a)p}A_p^p
\int_{-1}^1\varphi^{rp+2\a p}(x)\left|\int_{0}^{x}\varphi^{-(r+1)q-2\a q}(u)du\right|^{p/q}dx\\
&=:   2^{(\b-\a)p} A_p^p \Theta(\a,p).
\end{align*}
Note that
\[
\Theta(\a,1)   =\int_{-1}^1\varphi^{r+2\a }(x) \left( \sup_{u\in \left[ \min\{0,x\}, \max\{0,x\} \right]  } \varphi^{-r-1-2\a }(u)\right) dx.
\]
Recall that $r/2+\a\in J_p$ so that $rp+2\a p>-2$. We consider two cases.

{\sl Case} 1. Suppose that $rp+2\a p \ge -1$.
If $p=1$, then $r+2\a+1 \ge 0$ implies that
$\Theta(\a,1)   =  \int_{-1}^1 \varphi^{-1 }(x)    dx = \pi$, and
if $1<p<\infty$, then   $((r+1)q-1+2\alpha q)p/q=rp+2\a p+1\ge 0$,  and hence
\begin{align*}
\Theta(\a,p)&=
2\int_{0}^1\frac1{\varphi(x)}\left(\int_{0}^{x}\frac{\varphi^{(r+1)q-1+2\alpha q}(x)}{\varphi^{(r+1)q+2\a q}(u)}\,du\right)^{p/q}\,dx\\
&\le
2\int_{0}^1\frac1{\varphi(x)}\left(\int_{0}^{x}\frac{1}{\varphi(u)}\,du\right)^{p/q}\,dx
\le
2\int_{0}^1\frac{dx}{\varphi(x)}\left(\int_{0}^{1}\frac{du}{\varphi(u)}\right)^{p/q}\\
&=2(\pi/2)^{p}.
\end{align*}

{\sl Case} 2. Suppose now that $-2< rp+2\a p < -1$. If $p=1$, then
\[
\Theta(\a,1)   =\int_{-1}^1\varphi^{r+2\a }(x)   dx <\infty .
\]

If $1<p<\infty$, then  $(r+1)q+2\a q<1$. Hence
\[
\int_{0}^{1}\varphi^{-(r+1)q-2\a q}(u)du <  \int_{0}^{1}\varphi^{-1}(u)du  = \pi/2,
\]
and so
$$
\Theta(\a,p) \le 2 (\pi/2)^{p/q} \int_{0}^1\varphi^{rp+2\a p}(x)dx  < \infty.
$$
This completes the proof.
\end{proof}

We now show that, for a function $f\in\B^r_p(\wab)$, if  $r/2+\a\ge0$ and $r/2+\b\ge0$, then the modulus $\wkr(f^{(r)},t)_{\a,\b,\,p}$ is bounded.

\begin{lemma}\label{uppernorm} Let $k\in\N$, $r\in\N_0$, $r/2+\a\ge0$, $r/2+\b\ge0$, and $0< p\le\infty$. If $f\in\B^r_p(\wab)$, then
\be \label{modbounded}
\wkr(f^{(r)},t)_{\a,\b,\,p}\le c  \norm{ \wab \varphi^r f^{(r)}}{p}, \quad t>0,
\ee
where  $c$ depends only on $k$ and $p$.
\end{lemma}
\begin{proof}
In view of \ineq{larget}, we may limit ourselves to $t\le2/k$, and so $\Dom_{kh}\neq\emptyset$  if $0<h\le t$.
We set
$$u_i(x):=x+(i-k/2)h\varphi(x),\quad 0\le i\le k,
$$
and note that, for $x\in\Dom_{kh}$,
\begin{align*}
B_r(x)&:=\frac{\weight(x)}{w_{\a,\b}(u_i(x))\varphi^r(u_i(x))} \\ \nonumber
&=\left(\frac{1-u_i(x)-(k-i)h\varphi(x)}
{1-u_i(x)}\right)^{r/2+\a}\left(\frac{1+u_i(x)-ih\varphi(x)}{1+u_i(x)}\right)^{r/2+\b}\\&\le1\nonumber.
\end{align*}

 Therefore,
\begin{align*}
\lefteqn{   \norm{\weight(\cdot) f^{(r)}(u_i(\cdot))}{L_\infty(\Dom_{kh})}   }\\
&\quad =
\norm{ B_r(\cdot) w_{\a,\b}(u_i(\cdot))\varphi^r(u_i(\cdot))f^{(r)}(u_i(\cdot))}{L_\infty(\Dom_{kh})}\\
&\quad \le
\norm{\wab(u_i(\cdot))\varphi^r(u_i(\cdot))f^{(r)}(u_i(\cdot))}{L_\infty(\Dom_{kh})}\\
&\quad \le  \norm{\wab \varphi^rf^{(r)}}{\infty},
\end{align*}
that yields \ineq{modbounded} for $p=\infty$.

To apply the same arguments to the case $0< p<\infty$ we note
  that \ineq{derphi} yields $|\varphi'(x)| \le 1/(kh)$ for $x\in\Dom_{kh}$, so that
\[
u_i'(x)\ge 1 - |i-k/2|h |\varphi'(x)| \ge 1 - kh |\varphi'(x)|/2
\ge 1/2,\quad x\in\Dom_{kh},
\]
which implies
\[
\int_{\Dom_{kh}} |F(u_i(x))| dx\le2\int_{-1}^1 |F(u)|du ,
\]
for each   $F\in \L_1[-1,1]$.

Hence,
\begin{align*}
\norm{ \weight(\cdot) f^{(r)}(u_i(\cdot))}{L_p(\Dom_{kh})}^p
& \le
\norm{ w_{\a,\b}(u_i(\cdot))\varphi^r(u_i(\cdot))f^{(r)}(u_i(\cdot))}{L_p(\Dom_{kh})}^p\\
&\le
2\int_{-1}^1|\wab(x)\varphi^r(x)f^{(r)}(x)|^p dx\\
&=
2\norm{ \wab \varphi^rf^{(r)} }{p}^p.
\end{align*}
Thus,
\begin{align*}
\wkr(f^{(r)},t)_{\a,\b,\,p} &\le c \max_{0\le i\le k}  % 2^k \max\{1, 2^{1/p-1}\}
\norm{  \weight(\cdot) f^{(r)}(u_i(\cdot))}{L_p(\Dom_{kh})} %\\
  \le c  %2^{k+1/p}\max\{1, 2^{1/p-1}\}
\norm{ \wab \varphi^rf^{(r)}}{p},
\end{align*}
and the proof is complete.
\end{proof}

\begin{remark} \label{remlocal}
The same proof yields a local version of \ineq{modbounded} as well. Namely, for each $h>0$ and $[a,b]\subseteq \Dom_{kh}$,
\[
\norm{ \weight(\cdot) \Delta_{h\varphi(\cdot)}^k (f^{(r)},\cdot)}{\Lp[a,b]}  \le c \norm{ \wab \varphi^r f^{(r)}}{\Lp(S)},
\]
where $S := \left[ a - kh\varphi(a)/2, b+kh\varphi(b)/2\right]$.
\end{remark}

We now show that the modulus $\wkr(f^{(r)},t)_{\a,\b,p}$ may be infinite for a function $f\in\B^r_p(\wab)$ if either
$r/2+\a <0$ or $r/2+\b <0$.

When $p=\infty$, this is obvious. Indeed, suppose that $r/2+\b \ge 0$ and $-k \le r/2+\a<0$, and let $f(x) := (x-1)^{k+r}$. Then
$f\in\B^r_\infty(\wab)$ and $  \Delta_{h\varphi(x)}^k (f^{(r)},x)  \equiv c h^k \varphi^k(x)$. Hence,
$\weight(x)  \Delta_{h\varphi(x)}^k (f^{(r)},x)  \to \infty$ for $x$ such that $1-x-kh\varphi(x)/2 \to 0$. This implies that
$\wkr(f^{(r)},t)_{\a,\b,\infty}=\infty$, for all $t>0$. Note also that, by considering $f\in\C^r[-1,1]$ such that $f(x)=(1-|x|)^{k+r}$, $x\not\in[-1/2,1/2]$,  one can easily see that the same conclusion holds if  both $r/2+\a$ and $/2+\b$ are in $[-k,0)$.

When $p<\infty$, the arguments are not so obvious, but the conclusion is the same. The following theorem is valid.
\begin{theorem} \label{smallab}
Suppose that $k\in\N$, $r\in\N_0$, $\a\in\R$, $0< p < \infty$, and $r/2+\b <0$.
If $0<p<1$ and $r\geq 1$, we additionally assume that $r/2+\b <1-1/p$.
%\begin{itemize}
%\item[\rm (i)] $r/2+\b <0$, if $1\le p <\infty$,
%\item[\rm (ii)], if $0<p<1$.
%\end{itemize}
 Then there exists a function
 $f\in\B^r_p(\wab)$, such that, for all $t>0$,
\[
  \wkr(f^{(r)},t)_{\a,\b,p} = \infty .
\]
\end{theorem}

\begin{proof} %We fix  $k\in\N$, and
Let $\{\e_n\}_{n=0}^\infty$ be a decreasing sequence of positive numbers, tending to zero,  such that $\e_0<1/(2k)$ and
$$
(2+k)\e_n<\e_{n-1},\quad n\in\mathbb{N}.
$$
Define
%define for convenience $\e_n:=   (2k+1)^{-n}$ and
\[
J_n := \left[-1+ \e_n, -1+  \e_n(1+2^{-n}) \right] .
\]
Now,  let
 $f$ be such that
 \[
 f^{(r)}(x) :=
 \begin{cases}
 \left(x+1-\e_n  \right)^{-r/2-\b-1/p}, & \text{if $x\in J_n$  for some $n\in\N$,}\\
 0, & \text{otherwise} ,
 \end{cases}
 \]
Note that, in the case $r\geq 1$, since $-r/2-\b-1/p+1 >0$, the function $f^{(r-1)} (x) = \int_0^x f^{(r)}(u) du$ is locally absolutely continuous on $(-1,1)$.

Now,
\begin{align*}
2^{- |r/2+\a|p}\norm{\wab \varphi^r f^{(r)}}{p}^p & \le
 \sum_{n=1}^\infty \int_{J_n} |(1+x)^{r/2+\b} f^{(r)}(x)|^p dx \\
&
\le \sum_{n=1}^\infty \e_n^{(r/2+\b)p} \int_{J_n} |f^{(r)}(x)|^p dx \\
& =
  \sum_{n=1}^\infty \e_n^{(r/2+\b)p} \int_{0}^{\e_n2^{-n}} t^{-(r/2+\b)p-1} dt \\
& \le
c \sum_{n=1}^\infty   2^{(r/2+\b)np} <\infty .
\end{align*}
Hence, $f\in\B_p^r(\wab)$.

We now let
\[
x_n:=-1+\frac k2\e_n , \quad  h_n:=\frac{\e_n}{\varphi(x_n)}, \andd I_{k,n}:=[x_n,x_n+\e_n] ,
\]
so that
\[
\Dom_{kh_n}=[x_n,-x_n]\quad\text{and}\quad h_n<\sqrt{2\e_n}\to0,\quad n\to\infty.
\]
Since $\varphi(x)\ge\varphi(x_n)$, $|x|\le|x_n|$, we conclude that, for any $x\in I_{k,n} \subset [x_n,-x_n]$,
\begin{align*}
x-\left(\frac{k}2-2\right)h_n\varphi(x)&=x-\frac{k}2h_n\varphi(x)+2h_n\varphi(x)\ge-1+2h_n\varphi(x)\\
&\ge-1+2h_n\varphi(x_n)=-1+2\e_n>-1+\e_n(1+2^{-n}).
\end{align*}
Now, since $\varphi$ is concave and $\varphi(-1)=0$, we have
\[
\varphi(x_n+\e_n)<\frac{x_n+\e_n+1}{x_n+1}\varphi(x_n)=\left(1+\frac2k\right)\varphi(x_n),%\le2\varphi(x_n).
\]
and so, for all $x\in I_{k,n}$,
\begin{align*}
x+\frac{k}2h_n\varphi(x)&\le x_n+\e_n+\frac k2h_n\varphi(x_n+\e_n)\le
 x_n+\e_n+\left(1+\frac{k}2\right)h_n\varphi(x_n)\\
 &=-1+(2+k)\e_n<-1+\e_{n-1}.
\end{align*}
If $k\ge2$, this implies that, for all $2\le i \le k$ and $x\in I_{k,n}$,
$$
f^{(r)}(x+(i-k/2)h_n \varphi(x))=0  .
$$
Now,  denote
\[
y(x):=x+(1- k/2)h_n\varphi(x)
\]
and observe that
\be \label{derivy}
\frac 12 < y'(x)   < \frac 32 ,  \quad x \in [x_n, -x_n] ,
\ee
since, if $|x|\le |x_n|$, then it follows from \ineq{derphi} that
\be \label{auxder}
h_n| \varphi'(x)|  <   1/k
\ee
%\be \label{auxder}
%h_n| \varphi'(x)| =  \frac{ h_n |x|}{\varphi(x)} \le \frac{h_n |x_n|}{\varphi(x_n)}
%= \frac{\e_n|x_n|}{\varphi^2(x_n)} = \frac{2(x_n+1)|x_n|}{k\varphi^2(x_n)} < \frac 1k
%\ee
and so
\[
|y'(x)-1| \le \frac{k}{2} h_n | \varphi'(x)|    < \frac 12 .
\]

For all $k\in\N$, using $\norm{f_1+f_2}{p} \leq \max\{1, 2^{1/p-1}\}\left( \norm{f_1}{p} + \norm{f_2}{p}\right)$, we obtain
\begin{align*}
&2^{|\a+r/2|}\norm{ \wt_{kh_n}^{r/2+\a,r/2+\b}(\cdot)  \Delta_{h_n\varphi}^k (f^{(r)},\cdot)}{p}\\& \geq
2^{|\a+r/2|}\norm{\wt_{kh_n}^{r/2+\a,r/2+\b}(\cdot) \Delta_{h_n\varphi}^k (f^{(r)},\cdot)}{\Lp(I_{k,n})}\\
 & \ge
\norm{ (1+y(\cdot)-h_n\varphi(\cdot))^{r/2+\b}  \left( f^{(r)}(y(\cdot)-h_n\varphi(\cdot)) - k f^{(r)}(y(\cdot) )\right)  }{\Lp(I_{k,n})} \\
&\geq
k \min\{1, 2^{1-1/p}\}  \norm{ (1+y(\cdot)-h_n\varphi(\cdot))^{r/2+\b} f^{(r)}(y(\cdot))}{\Lp(I_{k,n})} - c \norm{\wab \varphi^r f^{(r)}}{p}\\
&\geq
k \min\{1, 2^{1-1/p}\} \norm{ (1+y(\cdot)-\e_n)^{r/2+\b} f^{(r)}(y(\cdot))}{\Lp(I_{k,n})} - c \norm{\wab \varphi^r f^{(r)}}{p},\\
 \end{align*}
where, in the second last inequality, we used the fact that $y'(x) -h_n \varphi'(x)= 1- kh_n\varphi'(x)/2 \sim 1$ that follows from \ineq{auxder}, and
in the last inequality, we used  that $r/2+\b<0$ and that $\e_n\le h_n\varphi(x)$ for all $x\in[x_n,-x_n]$.

In order to complete the proof, we show that
\[
H:=  \norm{ (1+y(\cdot)-\e_n)^{r/2+\b} f^{(r)}(y(\cdot))}{\Lp(I_{k,n})} = \infty .
\]
Assume to the contrary  that $H<\infty$. Since
$$
y(x_n)=-1+\e_n\le y(x)<-1+\e_{n-1},\quad x\in I_{k,n},
$$
there is a positive number $a_n\le \e_n$, such that
$$
f^{(r)}(y(x))=(1-\e_n+y(x))^{-r/2-\b -1/p},\quad x\in[x_n,x_n+a_n].
$$
Therefore,
$$
H^p
\ge \int_{x_n}^{x_n+a_n}
(1-\e_n+y(x) )^{-1} dx.
$$
Using the change of variable $v=u(x):=1-\e_n+y(x)$ and \ineq{derivy}
we get
$$
H^p
\ge \frac23\int_{x_n}^{x_n+a_n}
(u(x) )^{-1}u'(x) dx= \frac23\int_{0}^{u(x_n+a_n)}\frac{dv}v=\infty,
$$
that contradicts  our assumption $H<\infty$.

Thus, we have found a sequence $\{h_n\}_{n=0}^\infty$ of positive numbers, tending to zero, such that
$  \norm{ \wt_{kh_n}^{r/2+\a, r/2+\b}
\Delta_{h_n\varphi}^k (f^{(r)},\cdot)}{p}=\infty$, for all $n\in\mathbb{N}$. This means that $\wkr(f^{(r)},t)_{\a,\b,p} = \infty$, for all $t>0$.
\end{proof}

We now state some properties of the Jacobi weights that we  need in several proofs below.

\begin{proposition} \label{propxu}
For any $\a,\b\in\R$, $x\in\Dom_{2\delta}$ and $u\in[x-\delta\varphi(x)/2,x+\delta\varphi(x)/2]$,
\be \label{wabxu}
2^{-|\a|-|\b|} \wab(u) \leq \wab(x) \leq 2^{|\a|+|\b|} \wab(u),
\ee
in particular,
\be \label{varphixu}
\varphi(u)/2 \leq \varphi(x) \leq 2\varphi(u).
\ee
Also,
\be \label{weightxu}
 2^{-|\a|-|\b|} \wab(x) \leq \wt_\delta^{\a,\b}(x) \leq  2^{|\a|+|\b|} \wab(x), \quad x\in\Dom_{2\delta}.
\ee
\end{proposition}

\begin{proof}
For $x\in\Dom_{2\delta}$ and
$u\in[x-\delta\varphi(x)/2,x+\delta\varphi(x)/2]$, we have
\[
 (1-u)/2\le (1-x+\delta\varphi(x)/2)/2\le1-x\le2(1-x-\delta\varphi(x)/2)\le2(1-u)
\]
and
\[
 (1+u)/2\le (1+x+\delta\varphi(x)/2)/2\le 1+x\le 2(1+x-\delta\varphi(x)/2)\le2(1+u).
\]
This immediately yields \ineq{wabxu}.
Now,
\begin{align*}
\wt_\delta^{\a,\b}(x)
& =
w_{\a,0} (x+\delta\varphi(x)/2) w_{0,\b} (x-\delta\varphi(x)/2) \\
&\leq
2^{|\a|} w_{\a,0} (x) 2^{|\b|}w_{0,\b}(x) = 2^{|\a|+|\b|} \wab(x)
\end{align*}
and
\begin{align*}
\wab(x) &=  w_{\a,0} (x)  w_{0,\b}(x) \leq 2^{|\a|} w_{\a,0} (x+\delta\varphi(x)/2)  2^{|\b|}w_{0,\b} (x-\delta\varphi(x)/2) \\
&=
2^{|\a|+|\b|}\wt_\delta^{\a,\b}(x)
\end{align*}
complete the proof.
\end{proof}

\begin{lemma}\label{lem1}
If $k\in\N$, $r\in\N_0$, $r/2+\a\geq 0$, $r/2+\b \geq 0$, $0< p < \infty$ and
 $f\in\B^r_p(\wab)$, then
$$
\lim_{t\to0^+} \wkr(f^{(r)},t)_{\a,\b,p} = 0 .
$$
\end{lemma}

\begin{proof} Let $\epsilon>0$. For convenience, denote $C_p :=  \max\{1, 2^{1/p-1}\}$.   Since $f\in\B^r_p(\wab)$, there is $\delta>0$  such that
\[
\norm{ \wab \varphi^r f^{(r)}}{\Lp( [-1,1]\setminus \Dom_{\delta})} <  \frac\epsilon{2c_0 C_p} ,
\]
where $c_0$ is the constant $c$ from the statement of \lem{uppernorm}.
Set
\[
g^{(r)}(x):=\begin{cases} f^{(r)}(x),&\quad\text{if}\quad
x\in\Dom_{\delta},\\
0,&\quad\text{otherwise,}\end{cases}
\]
and note that,
since $g^{(r)}\in L_p[-1,1]$, there exists $t_0>0$ such that
\[
  \omega_k^\varphi(g^{(r)},t)_p< \epsilon/(2^{|\a-\b|+1}C_p),\quad 0<t\le t_0.
\]
Using \lem{uppernorm} and the fact that, if $r/2+\a, r/2+\b \geq 0$ and $x\in\Dom_{kh}$, then $\weight(x)\leq2^{|\a-\b|}$, we have
\begin{align*}
\wkr(f^{(r)},t)_{\a,\b,p} &\le C_p\wkr(g^{(r)},t)_{\a,\b,p} + C_p\wkr(f^{(r)}-g^{(r)},t)_{\a,\b,p} \\
& \le
2^{|\a-\b|}C_p \omega_k^\varphi(g^{(r)},t)_p + c_0 C_p  \norm{ \wab \varphi^r \left( f^{(r)}-g^{(r)} \right) }{p} \\
&<
\epsilon/2 + c_0 C_p  \norm{ \wab \varphi^r   f^{(r)}   }{\Lp( [-1,1]\setminus \Dom_{\delta})} \\
&\le   \epsilon,
\end{align*}
if $0<t\le t_0$.
 This completes the proof.
\end{proof}

We now turn our attention to the case $p=\infty$. It is clear that, in order for $\lim_{t\to0^+} \wkr(f^{(r)},t)_{\a,\b,\infty} = 0$ to hold we certainly need that $f\in\C^r(-1,1)$, but this condition is not sufficient. If $f\in\B^r_\infty(\wab)\cap \C^r(-1,1)$ and $r/2+\a, r/2+\b \geq 0$,  then we can only conclude that $\wkr(f^{(r)},t)_{\a,\b,\infty} < \infty$ for $t>0$. For example, if at least one of $r/2+\a$ and $r/2+\b$ is not zero, and $f$ is such that $f^{(r)}(x) := \wab^{-1}(x)\varphi^{-r}(x)$, $r\in\N_0$, then
$f\in\B^r_\infty(\wab) \cap\C^r(-1,1)$ and $\wkr(f^{(r)},t)_{\a,\b,\infty} \geq 1$.

\begin{lemma}\label{leminfinity}
If $k\in\N$, $r\in\N_0$, $r/2+\a\geq 0$, $r/2+\b \geq 0$,  and
 $f\in\B^r_\infty(\wab)\cap \C^r(-1,1)$, then
\be \label{maincase}
\lim_{t\to0} \wkr(f^{(r)},t)_{\a,\b,\infty} = 0
\ee
if and only if
\begin{itemize}
\item    Case $1$.  $r/2+\a>0$ and $r/2+\b>0$:
\be\label{case1}
\lim_{x\to\pm1}w_{\a,\b}(x)\varphi^r(x)f^{(r)}(x)=0.
\ee
\item Case $2$. $r/2+\a>0$ and $r/2+\b=0$:
\be\label{case2}
\lim_{x\to 1}w_{\a,\b}(x)\varphi^r(x)f^{(r)}(x)=0, \andd   f^{(r)} \in \C[-1,1).
\ee
\item Case $3$. $r/2+\a=0$ and $r/2+\b>0$:
\be\label{case3}
\lim_{x\to -1}w_{\a,\b}(x)\varphi^r(x)f^{(r)}(x)=0, \andd f^{(r)} \in \C(-1,1].
\ee
\item Case $4$. $r/2+\a=0$ and $r/2+\b=0$:
\be\label{case4}
\text{$f^{(r)} \in \C[-1,1]$}.
\ee
\end{itemize}
\end{lemma}

Note that  since, for $f\in B_\infty^r(\wab)$, $f^{(r)}$ may not be defined at $\pm 1$, when we write $f^{(r)} \in \C[-1,1)$, for example, we mean that $f^{(r)}$ can be defined at $-1$ so that it becomes continuous there.

\begin{proof}
Since $\wkr(f^{(r)},t)_{\a,\b,\infty} = \w_{k,0}^\varphi(g,t)_{r/2+\a, r/2+\b, \infty}$ with $g:= f^{(r)}$, without loss of generality, we may assume that $r=0$ throughout this proof. Note also that Case 4 is trivial since $\w_{k,0}^\varphi(f,t)_{0, 0, \infty} = \w^k_\varphi(f,t)_\infty$, the regular DT modulus, tends to $0$ as $t\to 0$ if and only if $f$ is uniformly continuous ($=$ continuous) on $[-1,1]$.

We now prove the lemma in Case 2, all other cases being similar.

Given $\e>0$, assume that \ineq{case2} holds, and let $\delta=\delta(\e) \in (0,1)$   be such that
\[
\wab(x) |f (x)|<2^{-k}\e,\quad x\in [1-\delta,1).
\]
Denote
\[
\omega(t):=\omega_k(f ,t;[-1,1-\delta/3]),
\]
the regular $k$th modulus of smoothness of $f$ on the interval $[-1,1-\delta/3]$, and note that
 $\lim_{t\to 0} \omega(t) = 0$ because of the continuity of $f$ on this interval. Thus, there exists $t_0>0$ such that  $t_0 \leq 2\delta/(3k)$ and $\omega(t_0)<\e/2^{\a}$, and we fix $0<h\le t_0$.

For $x\in\Dom_{kh}$, denote $J_x:=[x-kh\varphi(x)/2,x+kh\varphi(x)/2]\subseteq[-1,1]$.
If $x\le1-2\delta/3$, then $J_x\subseteq[-1,1-\delta/3]$. Hence,
\be\label{tmp3}
|\wt_{kh}^{\a,\b}(x)\Delta^k_{h\varphi(x)}(f,x)|\le2^{\a}|\Delta^k_{h\varphi(x)}(f,x)|<\e.
\ee
If, on the other hand, $x>1-2\delta/3$, then $J_x\subseteq[1-\delta,1]$. Hence, for some $\theta\in J_x$,
\begin{align}\label{tmp4}
|\wt_{kh}^{\a,\b}(x)\Delta^k_{h\varphi(x)}(f,x)|&\le 2^k \wt_{kh}^{\a,\b}(x)|f(\theta)|
 \le 2^k \wab(\theta) |f (\theta)|<\e.
\end{align}
Combining \ineq{tmp3} and \ineq{tmp4}, we get \ineq{maincase}.

Conversely, assume that $\a>0$, $\b=0$ and \ineq{maincase} holds. Observing that
$\lim_{t\to 0} \omega_k(f ,t;[-1,0])= 0$,
 we conclude that $f$ is uniformly continuous on $[-1,0]$, \ie $f\in\C[-1,1)$.
Also, given $\e>0$, fix $0<h<1/(2k)$  such that $\w_{k,0}^\varphi(f,h)_{\a,\b,\infty}<\e$. Let $x\in(3/4,1)$, and let $\theta\in(1/2,x)$ be such that
$\theta+kh\varphi(\theta)/2=x$. Then,
\[
|f(x)-\Delta^k_{h\varphi(\theta)}(f,\theta)|\le(2^k-1)\|f\|_{\C[0,1-h^2/4]}=:A_h,
\]
which yields
\begin{align*}
|\wab(x)f(x)|  %\wab(x) \left(|\Delta^k_{h\varphi(x)}(f,\theta)|+A_h\right)\\
&\le\frac{\wab(x)}{\wt_{kh}^{\a,\b}(\theta)}|\wt_{kh}^{\a,\b}(\theta)\Delta^k_{h\varphi(\theta)}(f,\theta)|+\wab(x) A_h\\
&\le \w_{k,0}^\varphi(f,h)_{\a,\b,\infty}+\wab(x) A_h .
\end{align*}
%where we applied the inequality $(1+x)\le3(1+\theta-kh\varphi(\theta)/2)$.
%
Hence,
$\limsup_{x\to1}|\wab(x) f (x)|\le \e$,
and so
$\lim_{x\to1}\wab(x) (x)f (x)=0$.
\end{proof}

%\sect{Equivalence of  moduli and $K$-functionals \& proof of the upper estimate in \thm{thm1.4}}\label{sec3}

\sect{Proof of the upper estimate in \thm{thm1.4}}\label{sec3}

We devote this section to proving that the moduli defined by \ineq{wkrdefinition} can be estimated from above by the appropriate $K$-functionals from Definition~\ref{kfunctionals}.

First, we need the following lemma.

\begin{lemma} \label{lem:estimate}
Let $k\in\N$, $r\in\N_0$, $r/2+\a\geq 0$, $r/2+\b\geq 0$ and $1\le p\le\infty$. If $g\in\B^{r+k}_p(\wab)$, then
\be\label{estimate}
\wkr(g^{(r)},t)_{\a,\b,\,p}\le c t^k \norm{ \wab \varphi^{k+r}g^{(k+r)}}{p} .
\ee
\end{lemma}

\begin{proof} We follow the lines of the proof of \cite{kls}*{Lemma 4.1} and rely on the calculations there, modified to accommodate the additional weight $w_{\a,\b}$.

 We begin with the well known identity
\be \label{diffid}
\Delta_h^k(F,x) = \int_{-h/2}^{h/2}\cdots
\int_{-h/2}^{h/2}F^{(k)}(x+u_1+\cdots+u_k)du_1\cdots du_k
\ee
and write
\begin{align*}
\lefteqn{\wkr(g^{(r)},t)_{\a,\b,\,p} =   \sup_{0<h\leq t} \norm{\weight  \Delta_{h\varphi}^k(g^{(r)},\cdot)}{\Lp(\Dom_{kh})}    }\\
&=
\sup_{0<h\leq t} \norm{\weight
\int_{-h\varphi/2}^{h\varphi/2}\cdots\int_{-h\varphi/2}^{h\varphi/2}
g^{(k+r)}(\cdot+u_1+\cdots+u_k)du_1\cdots du_k}{\Lp(\Dom_{kh})} \\
&\le
\sup_{0<h\leq t} \norm{
\int_{-h\varphi/2}^{h\varphi/2}\cdots\int_{-h\varphi/2}^{h\varphi/2}
\bigl( \wab \varphi^r |g^{(k+r)}|\bigr) (\cdot+u_1+\cdots+u_k)du_1\cdots du_k}{\Lp(\Dom_{kh})} ,
\end{align*}
where, in the last inequality, we used the fact that $r/2+\a\geq 0$ and $r/2+\b\geq 0$  implies
\[
\weight(x)\le \wab(v)\varphi^r(v), \quad \text{if}\quad x-kh\varphi(x)/2\le v\le x+kh\varphi(x)/2.
\]
By H\"older's inequality (with $1/p + 1/q=1$), for each $x\in\Dom_{kh}$ and
$|u| \le (k-1)h\varphi(x)/2$,
 we have
\begin{align*}
& \int_{-h\varphi(x)/2}^{h\varphi(x)/2} \bigl(\wab \varphi^r |g^{(k+r)}|\bigr)
(x+u+u_k)du_k
=\int_{x+u-h\varphi(x)/2}^{x+u+h\varphi(x)/2}
\bigl(\wab \varphi^r |g^{(k+r)}|\bigr) (v)dv  \\
 & \quad \le\|w_{\a,\b}\varphi^{k+r}g^{(k+r)}\|_{L_p(\mathcal{A}(x,u))}
\norm{ \varphi^{-k}}{L_q(\mathcal{A}(x,u))}\\
& \quad \le
\G_p^{\a,\b}(x; g, k, r)
\norm{ \varphi^{-k}}{L_q(\mathcal{A}(x,u))} ,
\end{align*}
where
\[
\mathcal{A}(x,u):=\left[x+u-\frac h2\varphi(x), x+u+\frac h2\varphi(x)\right]
\]
and
\[
\G_p^{\a,\b}(x; g, k, r) := \norm{w_{\a,\b}\varphi^{k+r}g^{(k+r)}}{L_p[x-kh\varphi(x)/2, x+kh\varphi(x)/2]} .
\]

Thus, the proof is complete, once we show that
\be \label{toshowp}
I(k,p) \leq ch^{k} \norm{ \wab \varphi^{k+r}g^{(k+r)}}{p} ,
\ee
where
\[
  I(k,p)
 :=    \norm{ \G_p^{\a,\b}(\cdot; g, k, r) \F_q(\cdot,k)
   }{\Lp(\Dom_{kh})} ,
\]
\begin{align*}
 \F_q(x,k) & :=    \int_{-h\varphi(x)/2}^{h\varphi(x)/2}
\cdots\int_{-h\varphi(x)/2}^{h\varphi(x)/2}
\norm{ \varphi^{-k}}{\Lq(\mathcal{A}(x,u_1+\cdots+u_{k-1}))}
du_1\cdots du_{k-1}, \quad \text{if}\,\, k\ge 2,
\end{align*}
and $\F_q(x,1) := \norm{ \varphi^{-1}}{L_q(\mathcal{A}(x,0))}$.

To this end, we write
\begin{align*}
\norm{\cdot}{\Lp(\Dom_{kh})} & \le \norm{\cdot}{\Lp(\Dom_{2kh})} + \norm{\cdot}{\Lp\left((\Dom_{kh}\setminus\Dom_{2kh})\cap[0,1]\right)}
+ \norm{\cdot}{\Lp\left((\Dom_{kh}\setminus \Dom_{2kh})\cap[-1,0]\right)} \\
&   =:I_1(p)+I_2(p)+I_3(p) .
\end{align*}
In order to estimate $I_1(p)$, using \ineq{varphixu},
 for $x\in\Dom_{2kh}$, we have
\[
 \F_q (x,k)
   \le
2^k  (h\varphi(x))^{k-1}    \varphi^{-k}(x) (h\varphi(x))^{1/q} \\
 =
2^k h^{k-1/p}\varphi^{-1/p}(x).
\]
Exactly the same sequence of inequalities as in \cite{kls}*{p. 141-142} with $\varphi^{k+r} g^{(k+r)}$ there replaced by
$\wab \varphi^{k+r}g^{(k+r)}$ yields the estimate
\[
I_1(p) \leq c h^{k} \norm{\wab \varphi^{k+r}g^{(k+r)}}{p}.
\]
We now estimate $I_2(p)$, the estimate of $I_3(p)$ being  analogous.
Denoting
\[
\E_{kh} := (\Dom_{kh}\setminus \Dom_{2kh})\cap[0,1]
\]
we note that, since $\G_p^{\a,\b}(x; g, k, r) \le \norm{\wab \varphi^{k+r}g^{(k+r)}}{p}$, $x\in \Dom_{kh}$, we are done if we show that
\be\label{kg}
\norm{ \F_q (\cdot, k)}{\Lp(\E_{kh})} \le c h^k.
\ee
It remains to observe that the estimates
\[
\int_{\E_{kh}}\left(\F_q(x, k)\right)^pdx
\le ch^{kp}
%\]
\andd
%\[
\sup_{x\in\E_{kh}}\F_1(x, k)\le ch^k
\]
which are, respectively, inequalities (4.19) and (4.10) from \cite{kls}, imply the validity of \ineq{kg}.
This completes the proof.
\end{proof}

\begin{lemma}\label{upper} Let $k\in\N$, $r\in\N_0$, $r/2+\a\geq 0$, $r/2+\b\geq 0$ and $1\le p\le\infty$. If $f\in\B^r_p(w_{\a,\b})$, then
\be\label{up}
\wkr(f^{(r)},t)_{\a,\b,\,p}\le c K^\varphi_{k,r}(f^{(r)},t^k)_{\a,\b,\,p}\,,\quad t>0.
\ee
\end{lemma}
\begin{proof}
%We may limit ourselves to $t\le2/k$, so that if $0<h\le t$, then $\Dom_{kh}\neq\emptyset$.
Take any $g\in\B^{r+k}_p(\wab)$.
Then, by \lem{hierarchy},
$g\in\B^r_p(\wab)$, and using Lemmas~\ref{uppernorm} and \ref{lem:estimate} we have
%\bc
%if $r\ge1$ or if $r=0$ and $1\le p<\infty$, and $g\in C(-1,1)\cap L_{w_{\a,\b}}^\infty$, if  $r=0$ and $p=\infty$.
%\ec
\begin{align*}
\wkr(f^{(r)},t)_{\a,\b,\,p} & \le\wkr(f^{(r)}-g^{(r)},t)_{\a,\b,\,p}+\wkr(g^{(r)},t)_{\a,\b,\,p} \\
&\le
  c \norm{ \wab \varphi^r \left( f^{(r)}-g^{(r)} \right) }{p}+  c t^k \norm{ \wab \varphi^{k+r}g^{(k+r)}}{p},
\end{align*}
 which immediately yields \ineq{up}.
\end{proof}

\sect{Equivalence of the moduli and Realization functionals \& proof of the lower estimate in \thm{thm1.4}}\label{seclower}

%\subsection{Special doubling and $A^*$ weights} \label{subd}

%We need a few results from \cites{k-infty, k-singular}.

%In this section, we establish lower estimates of $\wkr$ moduli via certain Realization functionals using some general results for special classes of doubling and $A^*$ weights.

In this section, using some general results for special classes of doubling and $A^*$ weights, we prove that, for all $0<p\le \infty$, the $\wkr$ moduli are equivalent to certain Realization functionals.
This, in turn, provides lower estimates of $\wkr$ by means of the appropriate $K$-functionals, thus proving the lower estimate in \thm{thm1.4}. This, of course, is meaningful only for $1\le p\le\infty$, as we recall that, for $0<p<1$, the $K$-functionals may vanish while the moduli do not.

For general definitions of doubling weights, $A^*$ weights, $\W(\Z)$ and $\W^*(\Z)$ see \cites{k-singular, k-infty}. We only mentioned that the Jacobi weights with nonnegative exponents belong to all of these classes (see \cite{k-infty}*{Remark 3.3} and \cite{k-singular}*{Example 2.7}).
We now restate some  definitions from \cites{k-infty, k-singular}, adapting them to the weights $\wab$ with $\a,\b\geq 0$, and state corresponding theorems for these weights only.

Let $\Z_{A,h}^1 := [-1,-1+Ah^2]$, $\Z_{A,h}^2 := [1-Ah^2,1]$ and $\I_{A, h}:= [-1+Ah^2, 1-Ah^2]$.

The main part weighted modulus of smoothness and the averaged main part weighted modulus   are defined, respectively,  as
 \[
 \Omega_\varphi^k(f, A, t)_{p,w}  :=   \sup_{0<h\leq t} \norm{w(\cdot)  \Delta_{h\varphi(\cdot)}^k(f, \cdot; \I_{A, h})}{\Lp(\I_{A, h})}
\]
and
\[
\widetilde{\Omega}_\varphi^k(f, A, t)_{p,w}  :=  \left( \frac{1}{t} \int_0^t \norm{w(\cdot)  \Delta_{h\varphi(\cdot)}^k(f, \cdot; \I_{A, h})}{\Lp(\I_{A, h})}^p dh \right)^{1/p} .
\]

The (complete)  weighted modulus of smoothness  and the (complete) averaged weighted modulus are defined as
\[
 \w_\varphi^k(f, A,   t)_{p,w}  :=  \Omega_\varphi^k(f, A, t)_{p,w}  + \sum_{j=1}^2  E_k(f, \Z_{2A,t}^j)_{w,p}  .
\]
and
\[
\widetilde{\w}_\varphi^k(f, A, t)_{p,w}  :=  \widetilde{\Omega}_\varphi^k(f, A, t)_{p,w}  + \sum_{j=1}^2  E_k(f, \Z_{2A,t}^j)_{w,p},
\]
respectively.

 The following is an immediate corollary of \cite{k-singular}*{Theorem 5.2} in the case $0<p<\infty$ and \cite{k-infty}*{Theorem 6.1} if $p=\infty$.

\begin{theorem}   \label{jacksonthm}
Let
  $k, \nu_0\in\N$, $\nu_0\geq k$, $0<p \le \infty$, $\a\geq 0$, $\b\geq 0$, $A>0$  and $f\in\Lpab$.
 Then, there exists $N\in\N$ depending on $k$, $\nu_0$, $p$, $\a$ and $\b$, such that
 for every $n \geq N$ and $\vartheta>0$,  there is a polynomial $P_n \in\Poly_n$ satisfying
\[
\norm{\wab(f-P_n)}{p}   \leq c   \widetilde{\w}_\varphi^k(f, A, \vartheta/n)_{p,\wab}  \leq c   \w_\varphi^k(f, A, \vartheta/n)_{p,\wab}
\]
and
\[
n^{-\nu} \norm{\wab \varphi^\nu P_n^{(\nu)}}{p}  \leq c   \widetilde{\w}_\varphi^k(f, A, \vartheta/n)_{p,\wab} \leq c   \w_\varphi^k(f, A, \vartheta/n)_{p,\wab}   ,      \quad k\leq \nu \leq \nu_0,
\]
where constants $c$  depend only on  $k$, $\nu_0$, $p$,   $A$, $\a$, $\b$ and $\vartheta$.
\end{theorem}

The following theorem is proved in \cite{stok}.

\begin{theorem} \label{thm:localwh}
Let $k\in\N$, $\a\geq 0$, $\b\geq 0$, $A>0$, $0< p\leq \infty$ and $f\in \Lpab$. Then, for any $0<t\leq \sqrt{2/A}$, we have
\be \label{localwh}
E_k(f, \Z_{A,t})_{\wab, p}  \le c \w_{k,0}^{*\varphi}(f,t)_{\a,\b,p}  \le c \w_{k,0}^\varphi(f,t)_{\a,\b,p} ,
\ee
where the interval $\Z_{A,t}$ is either $[1-At^2, 1]$ or $[-1,-1+At^2]$, and $c$ depends only on $k$, $p$, $\a$, $\b$ and $A$.

In particular, if $A=2$ and $t=1$, then
\be \label{ineq:regwh}
E_k(f)_{\wab,p} \le c \w_{k,0}^{*\varphi}(f, 1)_{\a,\b,p} \le c \w_{k,0}^\varphi(f, 1)_{\a,\b,p} .
\ee
\end{theorem}

We now show that the moduli ${\w}_\varphi^k(f, A, t)_{p,\wab}$ and $\widetilde{\w}_\varphi^k(f, A, t)_{p,\wab}$ may be estimated from above by the moduli
 $\w_{k,0}^{\varphi}(f,t)_{\a,\b,p}$ and $\w_{k,0}^{*\varphi}(f,t)_{\a,\b,p}$, respectively.

\begin{lemma} \label{lemma43}
Let $k \in\N$,  $\a\geq 0$, $\b\geq 0$, $A \geq 2k^2$  and $f\in\Lpab$, $0<p \le \infty$.
  Then, for $0<t\leq 1/\sqrt{A}$,
  \[
   {\w}_\varphi^k(f, A, t)_{p,\wab}  \leq c \w_{k,0}^{\varphi}(f,t)_{\a,\b,p}
  \]
  and
  \[
  \widetilde{\w}_\varphi^k(f, A, t)_{p,\wab}  \leq c \w_{k,0}^{*\varphi}(f,t)_{\a,\b,p} ,
  \]
  where constants $c$ depend  only on $k$, $p$, $\a$, $\b$ and $A$.
\end{lemma}

\begin{proof} Recall that $\I_{A, h} = [-1+Ah^2, 1-Ah^2]$ and note that, if $A \geq 2k^2$, then
\[
\I_{A, h} \subseteq \Dom_{2kh} \subset \Dom_{kh}  , \quad \text{for all $h>0$.}
\]
Since, by Proposition~\ref{propxu},
$\wab(x) \sim\wt_{kh}^{\a,\b}(x)$, $x\in\Dom_{2kh}$, we have
\[
\norm{\wab(\cdot)  \Delta_{h\varphi(\cdot)}^k(f, \cdot; \I_{A, h})}{\Lp(\I_{A, h})}
\leq c \norm{\wt_{kh}^{\a,\b}(\cdot)  \Delta_{h\varphi(\cdot)}^k(f, \cdot)}{\Lp( \Dom_{kh})},
\]
so that
\[
{\Omega}_\varphi^k(f, A, t)_{p,\wab} \leq c \w_{k,0}^{\varphi}(f,t)_{\a,\b,p}
\]
and
\[
\widetilde{\Omega}_\varphi^k(f, A, t)_{p,\wab} \leq c  \w_{k,0}^{*\varphi}(f,t)_{\a,\b,p} .
\]
Now, \thm{thm:localwh}   yields that, for $0<t\leq 1/\sqrt{A}$,
\begin{align*}
\lefteqn{ \max\left\{ E_k(f, [1-2At^2, 1])_{\wab, p},  E_k(f, [-1,-1+ 2At^2])_{\wab, p}  \right\} }\\
& \hspace{2cm} \le    c \w_{k,0}^{*\varphi}(f,t)_{\a,\b,p}  \le c \w_{k,0}^\varphi(f,t)_{\a,\b,p} ,
\end{align*}
and the proof is complete.
\end{proof}

The following is an immediate corollary of \thm{jacksonthm} and \lem{lemma43}.

\begin{corollary}   \label{jacksoncor}
Let $k \in\N$, $r\in\N_0$,    $r/2+\a\geq 0$, $r/2+\b\geq 0$  and $f\in\B_p^r(\wab)$, $0< p \leq \infty$.
 Then, there exists $N\in\N$ depending on $k$, $r$, $p$, $\a$ and $\b$, such that
 for every $n \geq N$ and $0<\vartheta \leq 1$,  there is a polynomial $P_n \in\Poly_n$ satisfying
\[
\norm{\wab \varphi^r (f^{(r)}-P_n^{(r)})}{p}   \leq c   \wkrav(f^{(r)}, \vartheta/n)_{\a,\b,p}  \leq c   \wkr(f^{(r)},   \vartheta/n)_{\a,\b,p},
\]
and
\[
n^{-k} \norm{\wab \varphi^{k+r} P_n^{(k+r)}}{p}  \leq c   \wkrav(f^{(r)}, \vartheta/n)_{\a,\b,p}  \leq c   \wkr(f^{(r)},  \vartheta/n)_{\a,\b,p}  ,
\]
where constants $c$  depend only on  $k$, $r$, $p$, $\a$, $\b$ and $\vartheta$.
\end{corollary}

Suppose now that $0<t\leq 2/k$, and $n\in\N$ is such that $n\geq N$ and $ c_1/t   \leq n \leq c_2/t$.  Then, denoting  $\mu := \max\{1, c_2\}$, \cor{jacksoncor} with $\vartheta = \min\{1, c_1\}$ implies that
 \begin{align} \label{mainlower}
K_{k,r}^\varphi(f^{(r)},t^k)_{\a,\b,p} & \leq   \mu^k K_{k,r}^\varphi(f^{(r)},(t/\mu)^k)_{\a,\b,p}
\leq \mu^k R_{k,r}^\varphi(f^{(r)},n^{-k})_{\a,\b,p} \\ \nonumber
&\leq   c \wkrav(f^{(r)}, \vartheta/n)_{\a,\b,p}
 \leq   c  \wkrav(f^{(r)}, t)_{\a,\b,p} \\ \nonumber
 &\leq   c  \wkr (f^{(r)}, t)_{\a,\b,p}.
\end{align}

Note that \ineq{mainlower} is valid for all $0<p\le \infty$. However, we remind the reader that, for $0<p<1$, the $K$-functional may become identically equal to zero.

Together with \lem{upper}, the sequence of estimates \ineq{mainlower} immediately yields \thm{thm1.4}.

We now show that the estimates in \lem{lemma43} may be reversed in some sense, \ie there exists $0<\theta\leq 1$ such that moduli $\w_{k,0}^{\varphi}(f,\theta t)_{\a,\b,p}$ and $\w_{k,0}^{*\varphi}(f,\theta t)_{\a,\b,p}$ may be estimated from above, respectively,  by   ${\w}_\varphi^k(f, A, t)_{p,\wab}$ and $\widetilde{\w}_\varphi^k(f, A, t)_{p,\wab}$.

\begin{lemma} \label{ab}
Let $k \in\N$,  $\a\geq 0$, $\b\geq 0$, $A>0$  and $f\in\Lpab$, $0<p \le \infty$.
  Then, there exists $0<\theta \le 1$ depending only on $k$  and $A$, such that  for all $0<t\leq \sqrt{1/A}$,
  \be \label{i1}
 \w_{k,0}^{\varphi}(f,\theta t)_{\a,\b,p} \le c  {\w}_\varphi^k(f, A, t)_{p,\wab}
  \ee
  and
  \be \label{i2}
\w_{k,0}^{*\varphi}(f,\theta t)_{\a,\b,p} \leq c  \widetilde{\w}_\varphi^k(f, A, t)_{p,\wab}  ,
  \ee
  where constants $c$ depend  only on $k$, $p$, $\a$, $\b$ and $A$.
\end{lemma}

\begin{proof}
Let $B := \max\{A^2, 4k^2\}$,  $\theta := \min\left\{1, \sqrt{A/(kB)}\right\}$, $0<t\leq \sqrt{1/A}$ and $0<h\leq \theta t$. Note that  $h\le \sqrt{1/B}$ and, if $x\in \I_{B,h}$, then $x\pm kh\varphi(x)/2 \in \I_{A,h}$.
Also, $\I_{B,h} \subset \Dom_{2kh}$, and so
Proposition~\ref{propxu} implies that
$\wab(x) \sim\wt_{kh}^{\a,\b}(x)$, for all  $x\in\I_{B,h}$.
Hence,
\be \label{au1}
\norm{\wt_{kh}^{\a,\b}(\cdot)  \Delta_{h\varphi(\cdot)}^k(f, \cdot)}{\Lp(\I_{B,h})}  \leq
c \norm{\wab(\cdot)  \Delta_{h\varphi(\cdot)}^k(f, \cdot; \I_{A,h})}{\Lp( \I_{A,h})}  .
\ee

Now, let   $S_1 :=  [0,1] \cap \left(\Dom_{kh}\setminus \I_{B,h}\right)$. Then, denoting $x_0 = 1-Bh^2$, we have
\[
\widetilde S_1 :=\bigcup_{x\in S_1} \left[x-kh\varphi(x)/2, x+kh\varphi(x)/2 \right]
=\left[ x_0-kh\varphi(x_0)/2, 1\right] \subset [1-2At^2,1] .
\]
It now follows by Remark~\ref{remlocal} that, for a polynomial of best weighted approximation $p_k\in\Poly_k$ to $f$ on $[1-2At^2,1]$,
\begin{align} \label{s1}
\norm{ \wt_{kh}^{\a,\b}(\cdot) \Delta_{h\varphi(\cdot)}^k (f ,\cdot)}{\Lp(S_1)}  &\le c \norm{ \wab (f-p_k)}{\Lp(\widetilde S_1)}\\ \nonumber
&\le c E_k(f, [1-2At^2,1])_{\wab,p},
\end{align}
where we used the fact that any $k$th difference of $p_k$ is identically zero.

Similarly, for $S_2 :=  [-1,0] \cap \left(\Dom_{kh}\setminus \I_{B,h}\right)$ and
\[
\widetilde S_2 :=\bigcup_{x\in S_2} \left[x-kh\varphi(x)/2, x+kh\varphi(x)/2 \right] \subset [-1, -1+2At^2] ,
\]
we have
\be \label{s2}
\norm{ \wt_{kh}^{\a,\b}(\cdot) \Delta_{h\varphi(\cdot)}^k (f ,\cdot)}{\Lp(S_2)}
\le c E_k(f, [-1,-1+2At^2])_{\wab,p}.
\ee
Therefore,  noting that $\Dom_{kh} = \I_{B,h} \cup S_1 \cup S_2$ and combining \ineq{au1} through \ineq{s2}, we have, for all $0<h \le \theta t$,
\begin{align*}
\norm{\wt_{kh}^{\a,\b}(\cdot)  \Delta_{h\varphi(\cdot)}^k(f, \cdot)}{\Lp( \Dom_{kh})}
 &\le
c \norm{\wab(\cdot)  \Delta_{h\varphi(\cdot)}^k(f, \cdot; \I_{A,h})}{\Lp( \I_{A,h})} \\
& \quad
 + c \sum_{j=1}^2  E_k(f, \Z_{2A,t}^j)_{w,p}.
\end{align*}
 Estimates \ineq{i1} and \ineq{i2} now follow, respectively, by taking supremum and by integrating with respect to $h$ over $(0, \theta t]$, and using the fact that $\theta \le 1$.
\end{proof}

Using Lemmas~\ref{lemma43} and \ref{ab} we immediately get \thm{pless1} as a corollary of the
 following result that follows from \cite{k-singular}*{Corollary 11.2}.

\begin{theorem} \label{thm46}
Let  $k\in\N$, $0<p<1$, $A>0$, $\a\geq 0$, $\b\geq 0$ and $f\in\Lpab$. Then there exist $N\in\N$ depending on $k$, $p$, $\a$ and $\b$, and $\vartheta>0$ depending on $k$, $p$, $A$, $\a$ and $\b$, such that, for any $\vartheta_1\in (0,\vartheta]$, $n\geq N$, $\vartheta_1/n\leq t \leq \vartheta/n$, we have
\[
R^\varphi_{k,0}(f,n^{-k})_{\a,\b,p} \sim \widetilde{\w}_\varphi^k(f, A, t)_{p,\wab} \sim {\w}_\varphi^k(f, A, t)_{p,\wab}.
\]
\end{theorem}

%\begin{corollary}
%Let  $k\in\N$, $0<p<1$, $\a\geq 0$, $\b\geq 0$ and $f\in\Lpab$. Then there exist $N\in\N$ and $\vartheta>0$ depending on $k$, $p$, $\a$ and $\b$, such that, for any $\vartheta_1\in (0,\vartheta]$, $n\geq N$, $\vartheta_1/n\leq t \leq \vartheta/n$, we have
%\[
%R^\varphi_{k,0}(f,n^{-k})_{\a,\b,p} \sim \w_{k,0}^{*\varphi}(f,  t)_{\a,\b,p}\sim \w_{k,0}^{*\varphi}(f,  t)_{\a,\b,p}.
%\]
%\end{corollary}

\sect{Weighted DT moduli \& alternative proof of the lower estimate via $K$-functionals}

In this section, we provide an alternative proof, in the case $1\le p \le \infty$,  of the lower estimate of the moduli $\wkr(f^{(r)},t)_{\a,\b,p}$ and $\wkrav(f^{(r)},t)_{\a,\b,p}$  by appropriate $K$-functionals, using certain weighted DT moduli.

We   denote the $k$th forward and the $k$th backward differences by
$\overrightarrow{\Delta}_h^k(f,x) := \Delta_h^k(f,x+kh/2)$ and
$\overleftarrow{\Delta}_h^k(f,x):=\Delta_h^k(f,x-kh/2)$, respectively.

Adapting the weighted DT moduli which were defined in \cite{dt}*{p. 218 and (8.2.10)} for a weight $w$ on $D:=[-1,1]$, we set for $f\in L_{w,p}$,
\begin{align} \label{dtmod}
\omega_{\varphi}^k(f,t)_{w,p}&:= \sup_{0<h\leq t}
\norm{w(\cdot) \Delta_{h\varphi}^k (f, \cdot)}{\Lp[-1+t^*,1-t^*]}\\ \nonumber
& \quad +\sup_{0<h\leq t^*}\norm{w(\cdot) \overrightarrow\Delta_h^k(f, \cdot)}{\Lp[-1,-1+12t^*]} \\ \nonumber
& \quad +
\sup_{0<h\leq t^*}\norm{w(\cdot)\overleftarrow\Delta_h^k (f,\cdot)}{\Lp[1-12t^*,1]},
\end{align}
where  $t^*:=2k^2t^2$.
The first term on the right in the above equation is called the main-part modulus and denoted by $\Omega_{\varphi}^k(f,t)_{w,p}$. Obviously, we have
$\Omega_{\varphi}^k(f,t)_{w,p}\le\omega_{\varphi}^k(f,t)_{w,p}$.

Next, the weighted $K$-functional was defined in \cite{dt}*{p. 55 (6.1.1)} as
\[
K_{k,\varphi}(f,t^k)_{w,p}:=\inf_{g\in\B^k_p(w)}    \{\|w(f-g)\|_{p}+t^k\|w\varphi^k g^{(k)}\|_{p}\},
\]
and we note that
\[
K_{k,\varphi}(f,t^k)_{\wab,p} = K^\varphi_{k,0}(f,t^k)_{\a,\b,p}   .
\]

It was shown in \cite{dt}*{Theorem 6.1.1} that, given an appropriate weight $w$ (all  Jacobi weights   with {\em nonnegative} exponents  are included), the weighted $K$-functional is equivalent to the weighted DT modulus of $f$. Namely, by \cite{dt}*{Theorem 6.1.1}, for $1\le p\le\infty$,
\[
M^{-1}\omega^k_\varphi(f,t)_{w,p}\le K_{k,\varphi}(f,t^k)_{w,p}\le M\omega^k_\varphi(f,t)_{w,p}\,,\quad 0<t\leq t_0,
\]
where $t_0$ is some sufficiently small constant.
Hence, in particular, if $\a,\b\ge0$, then
\be\label{ko}
 \omega^k_\varphi(f,t)_{w_{\a,\b},p}\sim K_{k,\varphi}(f,t^k)_{w_{\a,\b},p} ,\quad 0<t\leq t_0.
\ee

Note that, if $\alpha <0$ or $\beta<0$, then there are functions $f$ in $\Lpab$ for which $\omega_\varphi^k(f,\delta)_{\wab,p} = \infty$. Indeed, the following example was given in \cite{k-kmon} (see also \cite{dt}*{Remark 6.1.2 on p. 56}) and, in fact, it was the starting point for our counterexample in \thm{smallab}.
Suppose that $1\le p<\infty$ and that $\delta >0$ is fixed. If $f(x) :=  (x+1-\e)^{-\b-1/p} \chi_{[-1+\e, -1+2\e]}(x)$ with  $\b<0$ and $0<\e<t^*$, then
$\norm{\wab f}{p} \leq c(\a,\b,p) $ (and so $f\in\Lpab$),  $\norm{\wab(\cdot) f(\cdot +\e)}{\Lp[-1,-1+12t^*]} = \infty$, and $\norm{\wab(\cdot) f(\cdot +i\e)}{p} = 0$, $2\leq i\leq k$, and therefore
\begin{align*}
   \sup_{0<h\leq t^*}\norm{\wab(\cdot) \overrightarrow\Delta_h^k(f, \cdot)}{\Lp[-1,-1+12t^*]}
  &\geq
\norm{\wab(\cdot) \overrightarrow\Delta_\e^k(f, \cdot)}{\Lp[-1,-1+12t^*]}  \\
&  =   \norm{\wab(\cdot) \left[ f(\cdot) - k f(\cdot+\e)\right]}{\Lp[-1,-1+12t^*]} \\
& = \infty.
\end{align*}
Also, if $f\in\Lpab$ then  (choosing $g\equiv 0$) we have
$K_{k,\varphi}(f,t^k)_{w_{\a,\b},p} \leq \norm{\wab f}{p} <\infty$.
Hence, \ineq{ko} is not  valid if $\a<0$ or $\b<0$ (see also \thm{smallab} with $r=0$).

An equivalent averaged weighted DT modulus
\begin{align} \label{dtaveraged}
\omega_{\varphi}^{*k}(f,t)_{w,p} & :=  \left(
\frac1t \int_0^t \int _{-1+t^*}^{1-t^*} |w(x) \Delta^k_{\tau
\varphi(x)}(f,x)|^p \, dx\, d\tau \right)^{1/p}\\ \nonumber
&\quad  +
\left( \frac1{t^*} \int_0^{t^*} \int_{-1}^{-1+At^*}  |w(x)
\overrightarrow{\Delta}^k_{u}(f,x)|^p \, dx\, du \right)^{1/p}\\
\nonumber
& \quad  + \left( \frac1{t^*} \int_0^{t^*} \int_{1-At^*}^1 |w(x)
\overleftarrow{\Delta}^k_{u}(f,x)|^p \, dx\, du \right)^{1/p},
\end{align}
where $1\leq p <\infty$,  $t^*:=2k^2t^2$, and $A$ is some sufficiently large absolute constant, was defined in \cite{dt}*{(6.1.9)}.
For $p=\infty$,
set $\omega_{\varphi}^{*k}(f,t)_{w,\infty}:=\omega_{\varphi}^k(f,t)_{w,\infty}$.
It was shown in \cite{dt}*{p. 57} that, for an  appropriate weight $w$ (again, all  Jacobi weights  with  nonnegative exponents  are included),  $1\le p\le\infty$ and sufficiently small $t_0>0$,
\be\label{aver}
K_{k,\varphi}(f,t^k)_{w,p}\le c \omega_{\varphi}^{*k}(f,t)_{w,p},\quad 0<t\leq t_0.
\ee

We now provide an alternative proof of
  the inverse estimate to \ineq{up} independent of the results in Section~\ref{seclower}. First, we need the following lemma.

\begin{lemma} \label{low} Let $k\in\N$, $r\in\N_0$, $r/2+\a\geq 0$, $r/2+\b \geq 0$,  $1\le p<\infty$ and $f\in\B_p^r(\wab)$. Then
\[
\omega_{\varphi}^{*k}(f^{(r)},t)_{w_{\a,\b}\varphi^r,p}\leq
c(k,r,\a,\b)\wkrav(f^{(r)},c(k)t)_{\a,\b,\,p}\,,\quad 0<t\leq c(k).
\]
\end{lemma}
\begin{proof} The proof of this lemma is very similar to that of Lemma 6.1 in \cite{kls}, but we still provide all details here for completeness.
The three terms in the definition \ineq{dtaveraged} are to be estimated separately, but the second and third are similar, so we will estimate the first two.
Since
$\omega_{\varphi}^{*k}(f^{(r)},t)_{w_{\a,\b}\varphi^r,p} = \omega_{\varphi}^{*k}(g,t)_{w_{r/2+\a,r/2+\b},p}$
and
$ \wkrav(f^{(r)},t)_{\a,\b,p} = \w_{k,0}^{*\varphi}(g,t)_{r/2+\a, r/2+\b, p}$
with $g:= f^{(r)}$, without loss of generality, we may assume that $r=0$ throughout this proof.

Note that $t^* = 2k^2 t^2$ implies that $[-1+t^*, 1-t^*]\subset\Dom_{2kt}\subset \Dom_{2k\tau}$, $0\leq \tau\leq t$,  so that by \ineq{weightxu}
we have
\begin{align*}
\lefteqn{\frac1t\int_0^t\int_{-1+t^*}^{1-t^*}|w_{\a,\b}(x) \Delta^k_{\tau\varphi(x)}
(f ,x)|^p\,dx\,d\tau}\\
%&\leq &
%\frac{2^{(\a+\b+r)p}}{t}
%\int_0^t\int_{-1+t^*}^{1-t^*}| \wt^{r/2+\a,r/2+\b}_{k\tau}(x)\Delta^k_{\tau
%\varphi(x)}(f^{(r)},x)|^p \, dx\, d\tau \\
&\quad \leq
\frac{2^{(\a+\b)p}}{t}\int_0^t \int_{\Dom_{2k\tau}}| \wt^{\a,\b}_{k\tau}(x)
\Delta^k_{\tau\varphi(x)}(f,x)|^p\,dx\,d\tau\\
&\quad \leq
2^{(\a+\b)p}\w_{k,0}^{*\varphi}(f,t)_{\a,\b,\,p}^p\,.
\end{align*}

In order to estimate the second term we follow the proof of \cite{kls}*{Lemma 6.1} and assume that $t\le(2k\sqrt{A+k/2})^{-1}$. Then
\begin{align*}
&\frac1{t^*}\int_0^{t^*}\int_{-1}^{-1+At^*}|w_{\a,\b}(x) \overrightarrow{\Delta}^k_{u}(f,x)|^p\,dx\,du\\
&=\frac1{t^*}\int_0^{t^*}\int_{-1}^{-1+At^*}|w_{\a,\b}(x) \Delta^k_{u}(f,x+ku/2)|^p\,dx\,du\\
&\le\frac1{t^*}\int_0^{t^*}\int_{-1+ku/2}^{-1+(A+k/2)t^*}|w_{\a,\b}(y-ku/2) \Delta^k_{u}(f,y)|^p\,dy\,du\\
&\le\frac1{t^*}\int_{-1}^{-1+(A+k/2)t^*}\int_0^{2(y+1)/k}|w_{\a,\b}(y-ku/2) \Delta^k_{u}(f,y)|^p\,du\,dy\\
&=
\frac1{t^*}\int_{-1}^{-1+(A+k/2)t^*}\int_0^{2(y+1)/(k\varphi(y))}\varphi(y)|w_{\a,\b}(y-kh\varphi(y)/2)
  \Delta^k_{h\varphi(y)}(f,y)|^p\,dh\,dy\\
&\le
c\frac1{t^*}\int_{-1}^{-1+(A+k/2)t^*}
\int_0^{2(y+1)/(k\varphi(y))}\varphi(y)| \wt_{kh}^{\a,\b}(y) \Delta^k_{h\varphi(y)}(f,y)|^p\,dh\,dy\\
&\le
c\frac1{\sqrt{t^*}}\int_{-1}^{-1+(A+k/2)t^*}\int_0^{2(y+1)/(k\varphi(y))}|\wt_{kh}^{\a,\b}(y)
\Delta^k_{h\varphi(y)}(f,y)|^p\,dh\,dy\\
&\le c
\frac1{\sqrt{t^*}}\int_0^{c\sqrt{t^*}}\int_{\Dom_{kh}\cap[-1,-1+(A+k/2)t^*]}
|\wt_{kh}^{\a,\b}(y) \Delta^k_{h\varphi(y)}(f,y)|^p \,dy\,dh\\
&\le c \w_{k,0}^{*\varphi} (f,c(k)t)_{\a,\b,p}^p\,,
\end{align*}
where for the third inequality we used the fact that, for  $y\le-1/2$ and $0\le h\le2(y+1)/(k\varphi(y))$,
\[
1-y+kh\varphi(y)/2 \leq 2 \left( 1-y-kh\varphi(y)/2 \right) ,
\]
and so
\[
w_{\a,\b}(y-kh\varphi(y)/2)  \leq 2^{\a} \wt_{kh}^{\a,\b}(y).
\]
This completes the proof.
\end{proof}

A similar proof yields (see \cite{kls}*{Lemma 6.2}) an analogous result in the case $p=\infty$.

\begin{lemma} \label{lowinf}
Let $k\in\N$, $r\in\N_0$, $r/2+\a\geq 0$, $r/2+\b\geq 0$  and $f\in\B_\infty^r(w_{\a,\b})$. Then
\[\omega_{\varphi}^{k}(f^{(r)},t)_{w_{\a,\b}\varphi^r,\infty} \le c(k,r,\a,\b)\wkr(f^{(r)},
c(k)t)_{\a,\b,\,\infty}\,,\quad0<t\le c(k).
\]
\end{lemma}

We are now ready to prove the inverse of the estimate \ineq{up}.

\begin{lemma}\label{lower} Let $k\in\N$, $r\in\N_0$, $r/2+\a\geq 0$, $r/2+\b\geq 0$ and $1\le p\le\infty$. If $f\in\B^r_p(w_{\a,\b})$, then
\be\label{lowineq}
K^\varphi_{k,r}(f^{(r)},t^k)_{\a,\b,\,p} \leq c  \wkrav(f^{(r)},t)_{\a,\b,\,p} \leq c \wkr(f^{(r)},t)_{\a,\b,\,p}\, ,\quad 0<t\leq 2/k.
\ee
%where $c=c(k,r,p,\a,\b)$.
\end{lemma}

\begin{proof}

Combining \ineq{aver} with the weight $w=\wab\varphi^r$  with Lemmas~\ref{low} and \ref{lowinf}, we obtain, for $1\leq p\leq \infty$,
 \begin{align*}
 K^\varphi_{k,r}(f^{(r)},t^k)_{\a,\b,\,p} & = K_{k,\varphi}(f^{(r)},t^k)_{\wab \varphi^r,\,p} \leq c \omega_{\varphi}^{*k}(f^{(r)},t)_{w_{\a,\b}\varphi^r,\,p} \\
 & \leq c \wkrav(f^{(r)},c(k)t)_{\a,\b,\,p}\, , \quad 0<t\leq c\,.
 \end{align*}
Hence, we have
\be \label{123}
K^\varphi_{k,r}(f^{(r)},t^k)_{\a,\b,\,p}\le c\wkrav(f^{(r)},c_1 t)_{\a,\b,\,p}\,,\quad 0<t\le c_2,
\ee
where $c_1$ and $c_2$ are some positive constants  that may depend only on $k$.

Suppose now that $0<t\leq 2/k$. Then, denoting  $\mu := \max\{1,c_1, 2/(kc_2) \}$ and   using \ineq{123} we have
 \begin{align*}
K_{k,r}^\varphi(f^{(r)},t^k)_{\a,\b,\,p} & \leq   \mu^k K_{k,r}^\varphi(f^{(r)},(t/\mu)^k)_{\a,\b,\,p} \leq c  \wkrav(f^{(r)}, c_1  t/\mu)_{\a,\b,p} \\
& \leq   c \wkrav(f^{(r)}, t)_{\a,\b,\,p}\, ,
\end{align*}
which is the first inequality in \ineq{lowineq}.
Finally, the second inequality in \ineq{lowineq} follows from \ineq{ineq}.
\end{proof}

\end{document}